\documentclass{article}
\usepackage{graphicx}
\usepackage{amsmath}
\usepackage{hyperref}
\usepackage[all,rotate,2cell]{xy}
\usepackage{amsfonts}
\usepackage{amssymb}
\usepackage{url}
\usepackage{verbatim}

\UseTwocells

\newtheorem{theorem}{Theorem}

\newtheorem{definition}[theorem]{Definition}
\newtheorem{example}[theorem]{Example}
\newtheorem{lemma}[theorem]{Lemma}
\newtheorem{proposition}[theorem]{Proposition}
\newtheorem{remark}[theorem]{Remark}
\newenvironment{proof}[1][Proof]{\textbf{#1.} }{\ \rule{0.5em}{0.5em}}

\newcommand{\catname}[1]{\mathcal{#1}}
\newcommand{\catA}{\catname{A}}
\newcommand{\catB}{\catname{B}}
\newcommand{\catC}{\catname{C}}

\newcommand{\catE}{\catname{E}}
\newcommand{\catF}{\catname{F}}

\newcommand{\baseS}{\catname{S}}             
\newcommand{\thT}{\mathbb{T}}                
\newcommand{\thU}{\mathbb{U}}                
\newcommand{\thH}{\mathbb{H}}                
\newcommand{\thone}{\mathrm{1\!\!1}}         
\newcommand{\thob}{\mathbb{O}}               
\newcommand{\thext}{\subset}                 
\newcommand{\theqext}{\Subset}               
\newcommand{\thmorph}{\lessdot}              
\newcommand{\thextmap}{\supset}              
\newcommand{\thmorphmap}{\gtrdot}            

\newcommand{\upstairs}[1]{\overline{#1}}    
\newcommand{\downstairs}[1]{\underline{#1}} 
\newcommand{\stairs}[1]{{#1}\mathord{\Downarrow}}

\newcommand{\Mod}[2]{{#2\text{-}\mathbf{Mod}\text{-}#1}}  
\newcommand{\Mods}[2]{{#2\text{-}\mathbf{Mod}_{s}\text{-}#1}}  
\newcommand{\Modnocat}[1]{{\mathbf{Mod}\text{-}#1}}  
\newcommand{\Modsnocat}[1]{{\mathbf{Mod}_{s}\text{-}#1}}  

\newcommand{\cod}{\mathop{\mathsf{cod}}}

\newcommand{\Id}{\mathop{\mathsf{Id}}}
\newcommand{\comp}{\mathsf{comp}}            
\newcommand{\id}{\mathsf{id}}                

\newcommand{\Sh}[1]{\mathcal{S}{#1}}        

\newcommand{\Set}{\mathbf{Set}}             
\newcommand{\Fr}{\mathbf{Fr}}               
\newcommand{\AU}{\mathbf{AU}}               
\newcommand{\AS}{\mathbf{AS}}               
\newcommand{\AUpres}[1]{\AU\langle #1 \rangle} 
\newcommand{\Con}{\mathfrak{Con}}           
\newcommand{\BCon}{\mathfrak{BCon}}         
\newcommand{\GRD}{\mathbf{GRD}}             
\newcommand{\DLS}{\mathbf{DLS}}             
\newcommand{\pt}{\mathrm{pt}}               
\newcommand{\Loc}{\mathbf{Loc}}             
\newcommand{\Top}{\mathfrak{Top}}             
\newcommand{\Topiso}{\Top_{\cong}}          
\newcommand{\Cat}{\mathbf{Cat}}             
\newcommand{\RCat}{\mathbf{RCat}}           
\newcommand{\BTop}{\mathfrak{BTop}}         
\newcommand{\GTop}{\mathfrak{GTop}}         
\newcommand{\GTopU}[1]{\GTop\text{-}#1}     
\newcommand{\GTopB}[1]{\GTop\text{-}(#1\thext #1)}
\newcommand{\GTopE}[1]{\GTop\text{-}#1}     
\newcommand{\TopB}[1]{\Top_{\cong}\text{-}#1} 

\newcommand{\power}{\mathcal{P}}            
\newcommand{\fin}{\mathcal{F}}              
\newcommand{\List}{\mathop{\mathsf{List}}}  

\newcommand{\funK}{\mathfrak{K}}            

\begin{document}

\title{Arithmetic universes and classifying toposes}
\author{Steven Vickers\\
School of Computer Science, University of Birmingham,\\
Birmingham, B15 2TT, UK\\
\texttt{s.j.vickers@cs.bham.ac.uk}}

\maketitle
\begin{abstract}
  Reasoning in the 2-category $\Con$ of \emph{contexts,}
  certain sketches for arithmetic universes (i.e. list arithmetic pretoposes; AUs),
  is shown to give rise to base-independent results of Grothendieck toposes,
  provided the base elementary topos has a natural numbers object.

  Categories of strict models of contexts $\thT$ in AUs are acted on strictly
  on the left by non-strict AU-functors and strictly on the right by context maps,
  and the actions combine in a strict action of a Gray tensor product.

  Any context extension $\thT_0 \thext \thT_1$ gives rise to a \emph{bundle}.
  For each point of $\thT_0$
  -- a model $M$ of $\thT_0$ in an elementary topos $\baseS$ with nno --
  its fibre is a generalized space,
  the classifying topos $\baseS[\thT_1/M]$ for the geometric theory
  $\thT_1/M$ of $\thT_1$-models restricting to $M$.
  This construction is ``geometric'' in the sense that for any geometric morphism
  $f\colon \baseS' \to \baseS$,
  the classifier $\baseS'[\thT_1/f^\ast M]$ is got by pseudopullback of $\baseS[\thT_1/M]$
  along $f$.

  This is treated in a fibrational way by considering a 2-category $\GTop$ of
  Grothendieck toposes (bounded geometric morphisms)
  fibred (as bicategory) over a 2-category $\Top_{\cong}$ of elementary toposes with nno,
  geometric morphisms, and natural isomorphisms.
  The notion of classifying topos as representing object for a split fibration
  is then fibred over variable base using fibrations ``locally representable''
  over a second fibration.

  Maths Subject Classification
    18B25; 
    18D05  
    18D30  
    18C30  
    03G30  

  Keywords: geometric theory, 2-fibration, sketch, Gray tensor

\end{abstract}

\section{Introduction}\label{sec:Intro}

Grothendieck tells us that a topos is a generalized topological space,
and one aspect of this is that for any geometric theory $\thT$,
the classifying topos $\baseS[\thT]$ serves as the ``space of models of $\thT$''.
This is all understood point-free,
but in a very grand way --
the topos has the \emph{sheaves} on the space (the continuous set-valued maps)
instead of the opens (the continuous Sierpinski-valued maps).
The generalization then is that toposes encompass spaces such as that of sets
(the object classifier) for which there are not enough opens.

For example, \cite{TopCat} followed this methodology by using toposes as generalized spaces
of some domains used in the denotational semantics of programming languages.
However, that paper also mentioned a concern over the role of the base elementary topos $\baseS$.
It both supplies the infinities available in the geometric theory $\thT$
and underlies the categorical construction of $\baseS[\thT]$.
In \cite{TopCat} the geometric theories used would be expressible in any $\baseS$,
as long as it had a natural numbers object (nno).

A similar example is the geometric theory of Dedekind sections of the rationals,
whose classifying topos is the category of sheaves over the point-free real line:
it just requires an nno in $\baseS$.
In fact it is very reasonable to assume an nno under this methodology,
since \cite[Theorem~B4.2.11]{Elephant1} it is necessary and sufficient for the existence of
an object classifier $\baseS[\thob]$ and thence all classifying toposes.

We make that a standing assumption for the present paper:
for us, \emph{every elementary topos $\baseS$ has nno.}

Then the generalized spaces, the Grothendieck toposes,
are relative to an understood base topos $\baseS$:
they are the bounded geometric morphisms with codomain $\baseS$.

\cite{TopCat} made an effort to reason in a constructive way that would allow variation of base $\baseS$,
but that raises the question of why $\baseS$ should be needed in the first place.
The paper proposed that much, perhaps (with care) all,
of the reasoning was valid for arithmetic universes (list arithmetic pretoposes)
and could then be transferred to the toposes using the fact~\cite{JoWraith:AlgThTop}
that every elementary topos with nno is an AU,
and, for every geometric morphism $f$ between them,
the inverse image functor $f^\ast$ is a (non-strict) AU-functor.%
\footnote{
  Actually, the paper recognized specific places where non-geometric reasoning,
  such as the use of exponentials,
  was used to prove geometric sequents.
  Transferring such arguments to AUs would be non-trivial,
  but~\cite{ArithInd} develops techniques for doing so in some situations.
}
The geometric reasoning has extrinsic infinities (specifically: for infinite disjunctions)
supplied by the base $\baseS$.
Instead, AUs would have intrinsic infinities supplied in a type-theoretic way,
with sorts such as the natural numbers,
and could then use existential quantification over those.

A typical example from~\cite{TopCat} is the following.
The paper describes a geometric theory $\mathrm{IS}$ whose models,
``information systems'', are the compact bases of strongly algebraic domains.
Its classifying topos $\baseS[\mathrm{IS}]$, is then treated as the space of information systems.
By taking the topos of sheaves for the ideal completion of the generic information system,
we get a localic geometric morphism $\baseS[\mathrm{IS}][\mathrm{idl}]\to\baseS[\mathrm{IS}]$,
and this can be thought of as the generic strongly algebraic domain.
(A simpler and more familiar example would be the generic local homeomorphism,
got by taking the object classifier $\baseS[\thob]$ and slicing out the generic object.
The slicing has the effect of adjoining a generic global element to the generic object.)
Then some constructions of domain theory were explained in terms of the toposes.
For example, solving domain equations $D\cong F(D)$ requires continuity properties on $F$,
and the paper shows how to use the assumption that $F$ is a geometric endomorphism
on $\baseS[\mathrm{IS}]$.

So, if the classifying topos $\baseS[\mathrm{IS}]$ is the space of information systems, what is $\baseS$?
The traditional choice is the category $\Set$ of classical sets,
and the power of classical reasoning that it provides is still apparently needed for some significant
calculations in topos theory.
However, for present purposes,
where we are interested in exploiting the power of non-classical geometric reasoning,
it looks a distasteful choice.

The ``arithmetic'' proposal is to replace the classifying \emph{topos} by a classifying \emph{AU},
$\AUpres{\mathrm{IS}}$.
For the purposes of categorical logic they are completely analogous.
Each is the ``mathematics generated by a generic information system'',
but the first is geometric mathematics (finite limits, $\baseS$-indexed colimits)
while the second is arithmetic mathematics (finite limits, finite colimits, list objects).
Moreover, whereas the classifying AU is constructed up to isomorphism by universal algebra,
the classifying topos is somewhat hand-crafted, up to equivalence,
using presheaves (to get $\baseS$-indexed colimits) and sheaves.
The paper~\cite{Vickers:AUSk} now provides a 2-category $\Con$ that can be thought of as a 2-category
of generalized spaces in this arithmetic sense.
Its objects, \emph{contexts,} are the arithmetic theories $\thT$,
and the assignment $\thT\mapsto\AUpres{\thT}$ is full and faithful (and contravariant on 1-cells).

Hence the arithmetic proposal is to work in $\Con$.
The aim of the present paper is to begin to show how results there can be translated into base-independent
results for classifying toposes,
along the lines of~\cite{TopCat}.

Another potential benefit of the arithmetic approach is that it gives better control
of \emph{strictness,}
at least when we restrict to the \emph{contexts} of~\cite{Vickers:AUSk}.
Section~\ref{sec:IndCatMod} here will show how to reconcile the strict AU-functors implicit in $\Con$
with the non-strict ones needed for change of semantic domain.
This will enable us in much of our working to gain the advantages of split fibrations.

After that we move on to examining classifying toposes,
our basic approach being to fibre the constructions over the category $\Top_{\cong}$ of base toposes
(elementary toposes with nno, geometric morphisms, and natural isomorphisms)
and thereby gain the base-independence.
A central construction is a 2-category $\GTop$ of Grothendieck toposes,
fibred over $\Top_{\cong}$,
with reindexing by pseudopullback.
According to~\cite[B4.2]{Elephant1},
a theory over $\baseS$ is an indexed category over the fibre $\BTop/\baseS$ of $\GTop$ over $\baseS$ --
we shall call these ``elephant theories''.
Then a classifying topos is a representing object.
For a context $\thT$ we find a corresponding indexed category over $\GTop$,
and show that it is ``locally representable'' (Definition~\ref{def:locRepr})
in the sense that the classifying toposes for different bases $\baseS$ transform by pseudopullback.
The fact (Theorem~\ref{thm:pspbClassifiers}) that they do transform this way may be new.

In Section~\ref{sec:TwoFib} we collect miscellaneous remarks on the 2-fibrational background.

In Section~\ref{sec:ExtSp} we examine classifying toposes for contexts.
In fact, we deal with a relativized version,
with a context extension map $U\colon \thT_1 \to \thT_0$ (given by $\thT_0 \thext \thT_1$).
If each context represents ``the space of its models'',
then we wish to view $U$ as a bundle:
over each model $M$ of $\thT_0$, the fibre over it is the ``space of models of $\thT_1$ that restrict to $M$''.
We shall show how these fibres can be represented as classifying toposes.

Now we fibre over pairs $(\baseS,M)$, where $M$ is a strict model of $\thT_0$ in $\baseS$.
We find a geometric (though not arithmetic in general) theory $\thT_1/M$ of models of $\thT_1$ restricting to $M$,
and it has a classifying topos $\baseS[\thT_1/M] \to \baseS$ (with its generic model).

One example is that mentioned earlier,
where $\thT_0$ is the context $\mathrm{IS}$ for strongly algebraic information systems,
and $\thT_1$ extends it with an ideal.
Then $\baseS[\thT_1/M]$ is the topos of sheaves for the ideal completion of $M$.
More generally, suppose $\thT_0$ is a context for the ``GRD-systems'' used in~\cite{PPExp} to present frames,
and $\thT_1$ extends it with a point of the corresponding locale.
Then, given a GRD-system $M$ in $\baseS$,
$\baseS[\thT_1/M]$ is the topos of sheaves for that locale.
We shall see that our ``local representability'' condition implies the
``geometricity of presentations'' of~\cite{PPExp}.

\subsection{Sketches for arithmetic universes}\label{sec:AUSk}
We summarize the sketch approach to arithmetic universes as set out in \cite{Vickers:AUSk}.
The sketches are roughly as in \cite{BarrWells:TTT:TAC},
with a reflexive graph of nodes and edges for objects and morphisms,
a set of ``commutativities'' to specify commutative triangles,
and ``universals'' (the cones and cocones) for finite limits and finite colimits
-- specifically: terminals, pullbacks, initials, pushouts.
In addition they have universals to specify list objects,
thus gaining an nno as $\List 1$.

In our \emph{sketch extensions} $\thT\thext\thT'$ such universals may be introduced
only for fresh objects, and hence in a definitional way.
A \emph{context} is then an extension of the empty sketch $\thone$.

In \emph{equivalence extensions} $\thT\theqext\thT'$,
everything fresh that is introduced must have been implicitly present already.
This includes composites of composable pairs of edges;
commutativities deducible from existing ones (e.g. by unit laws or associativities);
universals, fillins for universals and uniqueness of fillins;
and inverses for certain edges that must be isomorphisms because of the categorical properties of AUs
such as balance, stability and exactness.

\emph{Homomorphisms} $\thT\thmorph\thT'$ are structure-preserving homomorphisms
for the algebraic theory of sketches.
They translate nodes to nodes, edges to edges, commutativities to commutativities
and universals to universals.
The two kinds of extensions are special cases of this.

Next, we have a notion of \emph{object equalities} between nodes,
certain edges that include all identity edges
but can also arise as fillins when the same universal construction is applied to equal data.
We extend this to object equalities between edges,
when their domains have an object equality and so do the codomains,
and there are appropriate commutativities to make a commutative square;
and then we extend to object equalities between homomorphisms,
using object equalities between corresponding nodes and edges in the image.

Putting these together we get a category $\Con$ whose objects are contexts.
Its morphisms, context \emph{maps}, are the dual of context homomorphisms,
but subject to (i) those for equivalence extensions are invertible,
and (ii) object equalities become identity morphisms between actually equal objects.
Every map $\thT_0 \to \thT_1$ is an equivalence class of opspans of homomorphisms
$\thT_0 \theqext \thT'_0 \thmorphmap \thT_1$.

Notice that, for each of the special symbols $\thext$, $\theqext$ and $\thmorph$,
the narrow end is at the codomain for the corresponding reduction \emph{map}.

For each context $\thT$ there is also a context $\thT^{\to}$ for which a model is a pair of models of $\thT$,
together with a $\thT$-homomorphism between them.
These enable us to define 2-cells between maps, using maps $\thT_0 \to \thT_1^{\to}$,
and $\Con$ becomes a 2-category.
It has finite PIE-limits (Product, Inserter, Equifier)
and pullbacks of extension maps (the duals of the homomorphisms corresponding to extensions).

There is a full and faithful 2-functor from $\Con$ to the category $\AU_s$ of AUs and strict AU-functors,
contravariant on 1-cells,
that takes $\thT \mapsto \AUpres{\thT}$.

A central issue for models of sketches is that of \emph{strictness.}
The standard sketch-theoretic notion is non-strict:
for a universal, such as a pullback of some given opspan,
the pullback cone can be interpreted as any pullback of the opspan.
However, we could also seek strict models that use the canonical pullbacks
(in categories where they exist).
Strictness is essential for the universal algebra that generates $\AUpres{\thT}$,
but in general it is inconvenient.
Significant parts of the present paper are concerned with relating the strict and the non-strict.

Contexts are designed to give us good control over strictness, as summarized by the following proposition.

\begin{proposition}\label{prop:extReindex}
  Let $U\colon\thT_1\to\thT_0$ be an extension map in $\Con$,
  that is to say one deriving from an extension $\thT_0\thext\thT_1$.
  Suppose in some AU $\catA$ we have a model $M_1$ of $\thT_1$,
  a strict model $M'_0$ of $\thT_0$,
  and an isomorphism $\phi_0\colon M'_0 \cong M_1 U$.
  \[
    \xymatrix{
      {\thT_1}
        \ar@{->}[d]_{U}
      & {M'_1}
        \ar@{.>}[r]^{\phi_1}_{\cong}
        \ar@{.}[d]
      & {M_1}
        \ar@{.}[d]
      \\
      {\thT_0}
      & {M'_0}
        \ar@{->}[r]^{\phi_0}_{\cong}
      & {M_1 U}
    }
  \]

  Then there is a unique model $M'_1$ of $\thT_1$
  and isomorphism $\phi_1\colon M'_1 \cong M_1$ such that
  \begin{enumerate}
  \item
    $M'_1$ is strict,
  \item
    $M'_1 U = M'_0$,
  \item
    $\phi_1 U = \phi_0$, and
\item
    $\phi_1$ is equality on all the primitive nodes for the extension $\thT_0\thext\thT_1$.
  \end{enumerate}
\end{proposition}

The proof can be deduced from the strictness results in~\cite{Vickers:AUSk}.
In brief, it is reduced by induction to the case of simple extension steps in $\thT_0\thext\thT_1$.
Adjoining a primitive node, $M'_1$ and $\phi_1$ are determined by (4).
Adjoining a primitive edge, $M'_1$ and $\phi_1$ are determined by the need to make $\phi_1$
an isomorphism.
Adjoining a universal, $M'_1$ is determined by (1)
and $\phi_1$ by (3), as the unique fillin consistent with $\phi_0$.

In the case where $\thT_0$ is the empty context $\thone$,
we see the important corollary that for a context $\thT$
\emph{every model is uniquely isomorphic to a unique strict model
with which it agrees on all primitive nodes.}

Thus in topos theory,
where non-strict AU-functors are liable to transform strict models into non-strict ones,
we can regain strictness of models.

\begin{example}
  The Proposition does not hold for arbitrary context maps $H\colon \thT_1 \to \thT_0$.
  Consider the diagonal $\Delta\colon\thob\to\thob^2$ given by the context homomorphism
  that takes both generic nodes in $\thob^2$ to the generic node in $\thob$.
  If $X$ is a model of $\thob$, then $X\Delta = (X,X)$.
  If we can find $X_1 \cong X \cong X_2$ with $X_1 \neq X_2$,
  then $(X_1,X_2) \cong X\Delta$ without itself being a $\Delta$-reduct.
\end{example}

\section{Indexed categories of models}\label{sec:IndCatMod}

In this section we deal with categories of models of AU-contexts.
For each AU $\catA$ and AU-context $\thT$
we have a category $\Mod{\thT}{\catA}$ of models of $\thT$ in $\catA$,
and a full subcategory $\Mods{\thT}{\catA}$ of strict models.

We shall show that $\Mods{\thT}{\catA}$ is acted on strictly (on the right) by $\Con$,
and strictly (on the left) by $\AU$,
the category of AUs and \emph{non-strict} AU-functors.
This strict left action arises because $\thT$, a context,
has the strict model corollary of Proposition~\ref{prop:extReindex}:
applying a non-strict AU-functor gives us a non-strict model,
but we can then replace it by its strict isomorph.%
\footnote{
  In fact, the definitions of extension and context in~\cite{Vickers:AUSk}
  were made in anticipation of these results.
}
The left and right actions commute up to isomorphism,
which we express in Theorem~\ref{thm:GrayAction} as a category strictly indexed over the Gray tensor product.

Note that the context maps, between contexts $\thT$,
correspond to \emph{strict} AU-functors between the classifying AUs $\AUpres{\thT}$.
What we have done, therefore, is in effect to have strict and non-strict AU-functors
acting on the right and left respectively,
with the Gray tensor action representing the interplay between strict and non-strict.

One might wonder whether we could instead have focused on the non-strict models $\Mod{\thT}{\catA}$.
There is an obvious action on the left by $\AU$,
and an action on the right, by model reduction,
by the context maps that correspond to context homomorphisms.
Those left and right actions commute up to equality.
However, the right action does not extend strictly to arbitrary context maps:
this is because the maps for context equivalence extensions,
which are invertible in $\Con$, give only equivalences between model categories, not isomorphisms.
We prefer to work with the strict action on strict models.

In any case, the non-strict models of a context $\thT$ are the strict models of an extension $\thT'$.
For each node $X$ in $\thT$ introduced by a universal,
adjoin another copy $X'$ with edges and commutativities to make $X'\cong X$.

\begin{definition}
  Let $\catA$ be an AU and $\thT$ a context.
  Then $\Mods{\thT}{\catA}$ is the category of strict models of $\thT$ in $\catA$.
\end{definition}

\begin{lemma}
  For each arithmetic universe $\catA$, we can define a 2-functor
  \[
    \Mods{\bullet}{\catA} \colon \Con \to \Cat
  \]
  for which $\Mods{\bullet}{\catA}(\thT) = \Mods{\thT}{\catA}$.
\end{lemma}
\begin{proof}
  Since those models are in bijection with strict AU-functors from $\AUpres{\thT}$ to $\catA$,
  and we have a (full and faithful) 2-functor from $\Con$ to $\AU_s^{op}$,
  this extends to a 2-functor $\Mods{\bullet}{\catA}$ as desired.
\end{proof}

If $M$ is a strict model in $\Mods{\thT_0}{\catA}$ and $H\colon \thT_0 \to \thT_1$
is a context map, then we write $MH$ for $\Mods{H}{\catA}(M)$.
If $H$ is the dual of a context homomorphism then $MH$ is got by model reduction.
If $H$ is the inverse of the dual for an equivalence extension $\thT_0 \theqext \thT_1$,
then $MH$ is got by interpreting all the adjoined ingredients of $\thT_1$ in the unique strict way.

Now we fix $\thT$ and let $\catA$ vary.

\begin{definition}
  Let $f\colon \catA_0 \to \catA_1$ be an AU-functor,
  $\thT$ a context and $M$ a model in $\Mods{\thT}{\catA_0}$.
  Then we define $f^{\ast}M = \Mods{\thT}{f}(M)$ as follows.
  We first define $f\cdot M$ as the non-strict model got by applying $f$ to $M$.
  Then $f^{\ast}M$
  is (using Proposition~\ref{prop:extReindex}) the unique strict model of $\thT$ in $\catA_1$,
  isomorphic to $f\cdot M$ and equal to it on the primitive nodes of $\thT$.

  We extend this to 2-cells $\alpha\colon f_0 \to f_1$ by treating them as AU-functors from $\catA_0$
  to the comma AU $\catA_1 \downarrow \catA_1$.
  $\alpha^{\ast}M \colon f_0^{\ast}M \to f_1^{\ast}M$ is then calculated by pasting
  the following diagram.
  \[
    \xymatrix{
      {\catA_1}
      & & {\catA_0}
        \xtwocell[ll]{}_{f_0}^{f_1}{^\alpha}
      & {\AUpres{\thT}}
        \ar@{->}[l]^(0.7){M}
        \ar@{->}@/_2pc/[lll]_{f_0^{\ast}M}^(0.2){\cong}
        \ar@{->}@/^2pc/[lll]^{f_1^{\ast}M}_(0.2){\cong}
    }
  \]
\end{definition}

\begin{proposition}
  For each context $\thT$ we have a 2-functor
  \[
    \Mods{\thT}{\bullet} \colon \AU \to \Cat
  \]
  for which $\Mods{\thT}{\bullet}(\catA) = \Mods{\thT}{\catA}$
  and $\Mods{\thT}{\bullet}(f)(M) = f^{\ast}(M)$.
\end{proposition}
\begin{proof}
  The main point is that it is strictly functorial on 1-cells $f$.
  Suppose we have AU-functors
  \[
    \xymatrix{ {\catA_2} \ar@{<-}[r]^{f_1} & {\catA_1} \ar@{<-}[r]^{f_0} & {\catA_0} }
    \text{.}
  \]
  Then $f_1^{\ast} f_0^{\ast} M$ and $(f_0 f_1)^{\ast}M$ are both
  the unique strict model of $\thT$ in $\catA_2$ that is isomorphic to $f_1 \cdot f_0 \cdot M$
  and agrees with it on all the primitive nodes.

  After this, the rest follows by pasting diagrams.
\end{proof}

The equation $f_1^{\ast} f_0^{\ast} M = (f_0 f_1)^{\ast}M$ will seem notationally perverse
for morphisms in $\AU$, composed diagrammatically,
but it makes more sense for geometric morphisms,
where the AU-functor for $f$ is $f^\ast$.

\begin{definition}
  Suppose we have 1-cells $f\colon \catA_0 \to \catA_1$ in $\AU$ and $H\colon\thT_0\to\thT_1$ in $\Con$.
  Then we define a natural isomorphism $\Sigma_{f,H}$ as follows.
  \begin{equation}\label{eq:pseudonat}
    \xymatrix{
      {\Mods{\thT_0}{\catA_0}}
        \ar@{->}[rr]^{\Mods{H}{\catA_0}}
        \ar@{->}[d]_{\Mods{\thT_0}{f}}
      & & {\Mods{\thT_1}{\catA_0}}
        \ar@{->}[d]^{\Mods{\thT_1}{f}}
      \\
      {\Mods{\thT_0}{\catA_1}}
        \ar@{->}[rr]_{\Mods{H}{\catA_1}}
        \ar@{}[urr]|{\Sigma_{f,H}\Downarrow}
      & & {\Mods{\thT_1}{\catA_1}}
    }
  \end{equation}
  For each $M$ in $\Mods{\thT_0}{\catA_0}$, we define the isomorphism
  $\Sigma_{f,H}(M)\colon f^{\ast}(MH) \cong (f^{\ast}M)H$
  by pasting the following diagram.
  \[
    \xymatrix{
      {\catA_1}
      & {\catA_0}
        \ar@{->}[l]^{f}
      & {\AUpres{\thT_0}}
        \ar@{->}[l]_-{M}
        \ar@{->}@/^2pc/[ll]^{f^{\ast}M}_(0.2){\cong}
      & & {\AUpres{\thT_1}}
        \ar@{->}[ll]^{\AUpres{H}}
        \ar@{->}@/_2pc/[llll]_{f^{\ast}(MH)}^(0.2){\cong}
    }
  \]
  Naturality is clear.
\end{definition}

\begin{theorem}\label{thm:GrayAction}
  The two actions on $\Mods{\bullet}{\bullet}$ by $\AU$ and $\Con$,
  together with the pseudo-naturality isomorphisms $\Sigma_{f,H}$,
  make up a ``cubical functor'' from $\AU\times\Con$ to $\Cat$ in the sense of~\cite{Gurski:Coh3dimCatTh},
  and hence a 2-functor from the Gray tensor product $\AU\otimes\Con$ to $\Cat$.
\end{theorem}
\begin{proof}
  There are three conditions to be checked.
  The first two are that the squares~\eqref{eq:pseudonat} paste together correctly,
  either horizontally or vertically, for composition of 1-cells in either $\Con$ or $\AU$.
  The third is that it pastes correctly with 2-cells in $\Con$ and $\AU$.
  All are clear by pasting the appropriate isomorphisms from the definition of $f^{\ast}$.
\end{proof}

\begin{lemma}\label{lem:fMH}
  \begin{enumerate}
  \item
    If $H$ is an extension map (for $\thT_1 \thext \thT_0$) then
    $(f^\ast M)H = f^\ast(MH)$ for every $f$ and $M$,
    and $\Sigma_{f,H}(M)$ is the identity morphism.
  \item
    If $H$ is an equivalence extension map ($\thT_1 \theqext \thT_0$),
    then $(f^\ast M)H^{-1} = f^\ast(MH^{-1})$.
  \end{enumerate}
\end{lemma}
\begin{proof}
  (1)
  $f^\ast (MH)$ is the unique strict model of $\thT_1$ isomorphic to $f\cdot(MH)$
  and equal to it on all the primitive noes of $\thT_1$.

  On the other hand $(f^{\ast}M)H \cong (f\cdot M)H = f\cdot(MH)$
  and they are equal on all the primitive nodes of $\thT_1$
  because they are also primitive in the extension $\thT_0$.

  (2)
  Apply part (1) to $MH^{-1}$.
\end{proof}

\begin{example}
  Equality in Lemma~\ref{lem:fMH} can fail whenever $H$ involves a context homomorphism
  that maps primitive nodes to non-primitives.
  Consider the context $\thT$ with a single node $T$, declared terminal,
  and $H\colon \thT \to \thob$ given by the sketch homomorphism
  that takes the generic node $X$ in $\thob$ to $T$.

  If $M$ is the unique strict model of $\thT$ in $\catA$,
  then $MH$ simply picks out the canonical terminal object,
  and $(f^\ast M)H$ does the same in $\catA'$.
  $f^\ast(MH)$ picks out the image under $f$ of the canonical terminal in $\catA$.
\end{example}

\section{Remarks on 2-fibrations}\label{sec:TwoFib}
In the 2-functor $\Mods{\thT}{\bullet} \colon \AU_s \to \Cat$ we have already seen
a category strictly indexed over the 2-category $\AU^{op}_s$.
As we proceed, however, we shall encounter non-strict indexations, with pseudofunctors,
and for these we shall prefer a fibrational approach.

For the appropriate notion of 2-fibration we shall follow~\cite{Buckley:Fibred2CatBicat},
which defines 2-fibrations between 2-categories and between bicategories.
Note that, although we deal only with 2-categories,
and 2-functors between them,
we shall still need to use the bicategorical notion of fibration
once we go beyond strictly indexed categories.
The essential difference, for a 2-functor $P \colon \catE \to \catB$,
is that the properties characterizing a cartesian 1-cell $f\colon x \to y$ in $\catE$ are weaker.
Given $g\colon z \to y$ and $h\colon Pz \to Px$ with $h (Pf) = Pg$,
we can lift $h$ to $\hat{h} \colon z \to x$ but the corresponding triangle in $\catE$
commutes only up to isomorphism.
\[
  \xymatrix{
    & {z}
      \ar@{.>}[dl]_{\hat{h}}
      \ar@{}[d]|{\cong}
      \ar@{->}[dr]^{g}
    \\
    {x}
      \ar@{->}[rr]_{f}
    & {}
    & {y}
  }
  \quad
  \xymatrix{
    & {Pz}
      \ar@{->}[dl]_{h}
      \ar@{}[d]|{=}
      \ar@{->}[dr]^{Pg}
    \\
    {Px}
      \ar@{->}[rr]_{Pf}
    & {}
    & {Py}
  }
\]

\subsection{The fibred 2-category of Grothendieck toposes}\label{sec:GTop}
By ``Grothendieck topos'', we mean a \emph{bounded} geometric morphism from some elementary topos $\catE$
to some, understood, base elementary topos $\baseS$.%
\footnote{As always for us, our elementary toposes are assumed to have nnos.}
The 2-category of Grothendieck toposes over $\baseS$ is studied in~\cite[B4]{Elephant1}
as $\BTop/\baseS$.

A notable property of $\BTop/\baseS$ is that any geometric theory $\thT$
(geometric, that is, with respect to $\baseS$)
has a classifying topos $\baseS[\thT]$ that behaves in many respect as
``the space of models of $\thT$'';
indeed, the whole of $\BTop/\baseS$ may then be viewed in a (generalized) topological way:
0-cells are spaces, 1-cells (geometric morphisms) are maps,
and 2-cells are specializations.

Our interest in using arithmetic universes is to deal with theories $\thT$
that depend on the base $\baseS$ only to the extent that nnos are required to exist.
Our aim here will be to prove results about Grothendieck toposes
that are fibred over choice of base.

From the point of view of indexed categories, the key result~\cite[B3.3.6]{Elephant1}
is that bounded geometric morphisms can be pseudo-pulled-back along arbitrary geometric morphisms.%
\footnote{Note that, in 2-categorical contexts, \cite{Elephant1} consistently omits ``pseudo-''
-- see B1.1.}
Thus for any geometric morphism $f\colon\baseS_0 \to \baseS_1$ we get a reindexing
$f^\ast \colon \BTop/\baseS_1 \to \BTop/\baseS_0$.
This does not extend to arbitrary natural transformations $\alpha \colon f \to g$ unless
the Grothendieck toposes are restricted to fibrations or opfibrations over $\baseS$,
so instead we restrict the $\alpha$s at the base level to be isomorphisms.

We write $\Top_{\cong}$ for the 2-category of elementary toposes (with nno),
geometric morphisms and natural isomorphisms.

We now express $\baseS\mapsto\BTop/\baseS$ as a fibration.

\begin{definition}\label{def:BTop}
  The data for the 2-category $\GTop$ is defined as follows.

  A 0-cell is a bounded geometric morphism $p \colon \catE \to \baseS$.

  A 1-cell $f=(\upstairs{f},\stairs{f},\downstairs{f})$
  from $\xymatrix{{\catE_0} \ar@{->}[r]^{p_0} & {\baseS_0}}$
  to $\xymatrix{{\catE_1} \ar@{->}[r]^{p_1} & {\baseS_1}}$
  is a square
  \[
    \xymatrix{
      {\catE_0}
        \ar@{->}[r]^{\upstairs{f}}
        \ar@{}[dr]|{\stairs{f}}
        \ar@{->}[d]_{p_0}
      & {\catE_1}
        \ar@{->}[d]^{p_1}
      \\
      {\baseS_0}
        \ar@{->}[r]_{\downstairs{f}}
      & {\baseS_1}
    }
  \]
  in which $\stairs{f}\colon \upstairs{f} p_1 \to p_0 \downstairs{f}$ is an isomorphism.

  Given two such 1-cells, $f$ and $f'$ from $p_0$ to $p_1$,
  a 2-cell $\alpha\colon f\to f'$ is a pair of natural transformations
  $\upstairs{\alpha}\colon \upstairs{f} \to \upstairs{f}'$
  and $\downstairs{\alpha}\colon \downstairs{f} \to \downstairs{f}'$
  \[
    \xymatrix{
      {\catE_0}
        \xtwocell[rr]{}^{\upstairs{f}}_{\upstairs{f}'}{\upstairs{\alpha}}
        \ar@{}[ddrr]^{\stairs{f}}_{\stairs{f'}}
        \ar@{->}[dd]_{p_0}
      & & {\catE_1}
        \ar@{->}[dd]^{p_1}
      \\ \\
      {\baseS_0}
        \xtwocell[rr]{}^{\downstairs{f}}_{\downstairs{f}'}{\downstairs{\alpha}}
      & & {\baseS_1}
    }
  \]
  such that the obvious diagram of 2-cells commutes.
  Moreover, as mentioned earlier, we require $\downstairs{\alpha}$ to be an isomorphism.
\end{definition}

It is clear that $\GTop$ is a 2-category
\begin{proposition}\label{prop:GTopFib}
  There is a 2-functor $\GTop^{co}\to\Top_{\cong}^{co}$ that forgets all but the downstairs part.
  Although it is strict, we consider it as a homomorphsm of \emph{bicategories}
  for the purposes of~\cite[3.1]{Buckley:Fibred2CatBicat}.
  \begin{enumerate}
  \item
    A 1-cell is cartesian iff it is a pseudopullback square in $\Top$.
  \item
    A 2-cell $\alpha$ is cartesian iff $\upstairs{\alpha}$ is an isomorphism.
  \item
   The 2-functor is a fibration of bicategories.
  \end{enumerate}
\end{proposition}
\begin{proof}
  (1):
  This is essentially the same as the proof of the result for 1-categories,
  that for the codomain fibration $\cod \colon \catC^{\to} \to \catC$,
  a morphism for $\catC^{\to}$ is cartesian iff it is a pullback square in $\catC$.
  The conditions for pseudopullbacks and cartesian 1-cells both bring in the 2-cells in the same way.
  For the ``$\Rightarrow$'' direction, note that an arbitrary topos $\catE$ can be treated
  as a 0-cell in $\GTop$ using the identity geometric morphism.

(2):
  If $\upstairs{\alpha}$ is an isomorphism then so is the 2-cell $\alpha$,
  and it is then clearly cartesian.
  For the converse,
  suppose $\alpha \colon f \to g$ is a cartesian 2-cell.
  (Note that because we are going to dualize, $\alpha$ is really cocartesian in $\GTop$.)
  Downstairs, $\downstairs{\alpha}$ is invertible
  and so by lifting $\downstairs{\alpha}^{-1}$
  we get $\alpha' \colon g \to f$, with $\alpha \alpha' = \Id_f$.
  By considering $\Id_g$ and $\alpha' \alpha$ as lifts of $\Id_{\downstairs{g}}$
  we see that they are equal.

  (3)
  Cartesian lifting of 1-cells arises because, in $\Top$,
  pseudopullbacks of bounded geometric morphisms along arbitrary geometric morphisms
  always exist~\cite[B3.3.6]{Elephant1}.

  Cartesian lifting of 2-cells is easy --
  in fact we can ensure that the upstairs part of the lifted 2-cell is an identity.
\end{proof}

Of course, $\Top^{co}_{\cong} \cong \Top_{\cong}$,
so we could equally well consider $\GTop^{co}$ as fibred over $\Top_{\cong}$.

\subsection{Elephant theories}\label{sec:ElephTh}
Here we briefly summarize the account in~\cite[B4.2]{Elephant1} of classifying toposes,
over a fixed base topos $\baseS$.

Central to its treatment is the 2-category $\BTop/\baseS$,
which is just the fibre of our $\GTop$ over $\baseS$.
A 0-cell is a bounded geometric morphism $p\colon\catE\to\baseS$,
but we shall frequently suppress $p$ notationally.
Thus we write about it as a topos $\catE$ \emph{equipped with} $p$.
Similarly, a 1-cell is a geometric morphism equipped with a specified isomorphism in the
triangle over $\baseS$.

One is used to thinking of the concept of ``logical theory'' in syntactic terms,
but in~\cite[B4.2]{Elephant1} it is defined semantically in a very grand way:
as the category of all models.
Moreover, geometric theories are incomplete in general,
and for that reason it is not enough simply to specify all the models in $\baseS$ --
there may not be enough of them.
Instead one must look at all the models in all bounded $\baseS$-toposes.
I shall refer to these as ``elephant theories'',
partly to acknowledge their use in \cite{Elephant1},
but also to convey something of the sheer quantity of data needed to describe
one of these theories in complete detail.
Obviously in practice we try to use syntactic presentations,
and that is one of the aims of the AU methods.

\begin{definition}
  An \emph{elephant theory over $\baseS$} is an indexed category $\thT$ over $\GTop/\baseS$.
  Then an object of $\thT(\catE)$ is a ``model of $\thT$ in $\catE$''.
\end{definition}

In our applications, the elephant theories will be strict, 2-functors to $\Cat$.
For each AU-context $\thT$ we have a 2-category $\Mods{\thT}{\bullet}$, strictly indexed over $\Top$,
and it restricts to $\BTop/\baseS$,
with the geometric morphisms $p$ playing no role in the reindexing.
Also, each context map $H\colon \thT_0 \to \thT_1$
gives a corresponding indexed functor from $\thT_0$ to $\thT_1$ as elephant theories.

A particularly important example is the context $\thob$,
the object classifier, with $\thob(\catE) = \catE$.

Given an elephant theory $\thT$ over $\baseS$,
a \emph{geometric construct} on $\thT$ is an indexed functor from $\thT$ to $\thob$.

\begin{definition}
  Let $\thT_0$ be an elephant theory over $\baseS$.
  A \emph{geometric extension} of $\thT_0$ is a theory built, starting from $\thT_0$,
  by a finite sequence of the following ``simple'' steps from $\thT$ to $\thT'$.
  \begin{itemize}
  \item
    \emph{Simple functional extension:}
    Let $H_0,H_1\colon \thT \to \thob$ be two geometric constructs.
    Define the theory $\thT'$ whose models in $\catE$ are pairs $(M,u)$
    where $M$ is a model of $\thT$ in $\catE$ and $u\colon MH_0 \to MH_1$ is a morphism.
    A morphism from $(M,u)$ to $(M',u')$ is morphism $\phi\colon M \to M'$
    such that that following diagram commutes.
    \[
      \xymatrix{
        {MH_0}
          \ar@{->}[r]^{u}
          \ar@{->}[d]_{\phi H_0}
        & {MH_1}
          \ar@{->}[d]^{\phi H_1}
        \\
        {M'H_0}
          \ar@{->}[r]_{u'}
        & {M'H_1}
      }
      \text{.}
    \]
  \item
    \emph{Simple geometric quotient:}
    Let $\phi\colon H_0 \to H_1$ be a morphism of geometric constructs on $\thT$.
    $\thT'$ is the theory whose models in $\catE$ are those models of $\thT$ for which
    $\phi$ is an isomorphism;
    its morphisms are all $\thT$-morphisms.
  \item
    \emph{Simple extension by primitive object:}
    We define $\thT'(\catE) = \thT(\catE)\times \catE$.
    In other words, we may write $\thT'=\thT\times\thob$.
  \end{itemize}
  Then a \emph{geometric theory} over $\baseS$ is a geometric extension of $\thone$.
\end{definition}
Note that~\cite{Elephant1} does not define the general notion of geometric extension,
but simply that of geometric theory as an extension of $\thob^{n}$ (for some finite $n$)
by simple functional extensions and simple geometric quotients.
The two are equivalent,
because no harm is done if the primitive sorts are all adjoined at the start,
and doing this $n$ times to $\thone$ gives $\thob^{n}$.

If $\thT_1$ is a geometric extension of $\thT_0$,
then there is a theory morphism from $\thT_1$ to $\thT_0$ given by model reduction.

For future reference we prove the following result that does not appear to be in~\cite{Elephant1}.
\begin{proposition}\label{prop:elephantpb}
  In the category of elephant theories over $\baseS$ and indexed functors between them,
  geometric extensions can be pulled back along any morphism.
\end{proposition}
\begin{proof}
  The point is that we have a pullback, not a pseudopullback.

  Let $H\colon \thT_0 \to \thT_1$ be an indexed functor between elephant theories over $\baseS$,
  and let $\thT'_1$ be a geometric extension of $\thT_1$ with indexed functor $U\colon\thT'_1 \to \thT_1$
  defined by model reduction.
  We define the elephant theory $\thT'_0$ by argumentwise pullback of categories.
  \[
    \xymatrix{
      {\thT'_0(\catE)}
        \ar@{->}[r]
        \ar@{->}[d]
      & {\thT'_1(\catE)}
        \ar@{->}[d]^{U(\catE)}
      \\
      {\thT_0(\catE)}
        \ar@{->}[r]_{H(\catE)}
      & {\thT_1(\catE)}
    }
  \]
  Thus a model of $\thT'_0$ is a pair $(M^0,M^1)$ of models of $\thT_0$ and $\thT'_1$
  for which $M^0 H = M^1 U$.

  For reindexing along $f \colon \catF \to \catE$ (over $\baseS$),
  the naive attempt of $f^\ast (M^0,M^1) = (f^\ast M^0, f^\ast M^1)$ fails because we only have
  \[
    (f^\ast M^0)H \cong f^\ast(M^0 H) = f^\ast(M^1 U) = (f^\ast M^1)U
    \text{.}
  \]
  (The last equality can be readily checked for different kinds of simple geometric extension.)
  The trick then is to define $f^\ast (M^0,M^1)$ as $(f^\ast M^0, N^1)$ for some $N^1 \cong f^\ast M^1$
  whose $\thT_1$-reduct is $(f^\ast M^0)H \cong (f^\ast M^1)U$.

  It suffices to check the three kinds of simple geometric extension.
  For extension by primitive sort, $\thT'_1 = \thT_1\times\thob$,
  we find that $\thT'_0$ as defined by pullback is $\thT_0\times\thob$.
  For the reindexing question,
  we have $M^1$ of the form $(M^0 H, X)$ and define $N^1 = ((f^\ast M^0)H, f^\ast X)$.

  The next case is when $\thT'_1$ is a simple functional extension of $\thT_1$
  for two geometric constructs $G_0,G_1 \colon \thT_1 \to \thob$.
  We find that $\thT'_0$, as defined by pullback,
  is a simple functional extension of $\thT_0$ for $HG_0$ and $HG_1$.
  For the reindexing, we have $M^1$ of the form $(M^0 H, u\colon M^0 H G_0 \to M^0 H G_1))$.
  Then we take $N^1$ to be $((f^\ast M^0) H, u')$,
  where $u'$ is so as to make the following diagram commute.
  \[
    \xymatrix{
      {(f^\ast M^0)HG_0}
        \ar@{->}[r]^{\cong}
        \ar@{->}[d]_{u'}
      & {(f^\ast (M^0 H))G_0}
        \ar@{->}[r]^{\cong}
      & {f^\ast (M^0 H G_0)}
        \ar@{->}[d]^{f^\ast u}
      \\
      {(f^\ast M^0)HG_1}
        \ar@{->}[r]^{\cong}
      & {(f^\ast (M^0 H))G_1}
        \ar@{->}[r]^{\cong}
      & {f^\ast (M^0 H G_1)}
    }
  \]

  For the final case, $\thT'_1$ is an extension of $\thT_1$ by simple geometric quotient
  for a morphism $\phi\colon G_0 \to G_1$ of two geometric constructs on $\thT_1$.
  Now $\thT'_0$ is an extension of $\thT_0$ by simple geometric quotient
  for a morphism $H\phi\colon HG_0 \to HG_1$.
\end{proof}

\begin{definition}
Let $\thT$ be an elephant theory over $\baseS$.
A \emph{classifying topos} for $\thT$ is a bounded $\baseS$-topos
$p\colon \baseS[\thT] \to \baseS$,
equipped with a ``generic'' $\thT$-model $G$,
such that, for each bounded $\baseS$-topos $\catE$,
the functor
\[
  \GTop/\baseS[\catE,\baseS[\thT]] \to \thT(\catE)\text{,}
  \quad f \mapsto f^\ast G
  \text{,}
\]
is one half of an equivalence of categories.
\end{definition}

Since all our toposes have nno, \cite[Theorem~B4.2.9]{Elephant1} tells us that
\emph{every geometric theory has a classifying topos.}

\subsection{Representability}\label{sec:Represent}
Classifying topos as above is \emph{defined} in terms of representability of an indexed category.
We now look at how this appears in terms of fibrations.

Suppose $\catC$ is a 2-category, and $F\colon\catC^{coop}\to\Cat$ a pseudofunctor.
Its Grothendieck construction is made, according to~\cite[3.3.3]{Buckley:Fibred2CatBicat},
as follows.
(We are in a somewhat simpler situation.
We have not allowed $\catC$ to be a bicategory, nor have we allowed non-trivial 2-cells in each $F(X)$.
When $F$ is strict -- as in fact it is in our applications --
then we can apply~\cite[2.2]{Buckley:Fibred2CatBicat} to get a fibration of 2-categories.)

The fibred bicategory $\catE$, actually a 2-category, though not fibred as such, has --

\emph{0-cells} are pairs $(x,x_{\_})$ of objects of $\catC$ and $Fx$.

\emph{1-cells} are pairs $(f,f_{\_}) \colon (x,x_{\_}) \to (y,y_{\_})$
where $f\colon x \to y$ and $f_{\_} \colon x_{\_} \to Ff(y_{\_})$.

\emph{2-cells} $(f,f_{\_}) \to (g,g_{\_}) \colon (x,x_{\_}) \to (y,y_{\_})$
are 2-cells $\alpha\colon f \to g$ such that the following diagram commutes.
\[
  \xymatrix{
    {x_{\_}}
      \ar@{->}[rr]^{f_{\_}}
      \ar@{->}[dr]_{g_{\_}}
    & {}
      \ar@{}[d]|{=}
    & {Ff(y_{\_})}
    \\
    & {Fg(y_{\_})}
      \ar@{->}[ur]_{F\alpha_{y_{\_}}}
  }
\]

Then the 1-cell $(f,f_{\_})$ is cartesian iff $f_{\_}$ is an isomorphism,
and the 2-cell $\alpha$ is cartesian iff it is an isomorphism.

In the following proposition we characterize representability of the pseudofunctor $F$
in a purely fibrational way, independent of $F$ as choice of cleavage.

\begin{proposition}\label{prop:psfunctorRepr}
  Let $F\colon\catC^{coop}\to\Cat$ be a pseudofunctor as above,
  and let $P\colon\catE\to\catC$ be its Grothendieck construction.
  Then $F$ is representable iff there is an object $(x,x_{\_})$ in $\catE$
  (a representing object) with the following properties.
  \begin{enumerate}
  \item\label{eq:psfunctorReprCondOne}
    For each $(y,y_{\_})$ in $\catE$,
    there is a cartesian 1-cell $(f,f_{\_}) \colon (y,y_{\_}) \to (x,x_{\_})$.
  \item\label{eq:psfunctorReprCondTwo}
    Each cartesian 1-cell $(f,f_{\_}) \colon (y,y_{\_}) \to (x,x_{\_})$ is terminal
    in $\catE((y,y_{\_}), (x,x_{\_}))$.
  \end{enumerate}
\end{proposition}
\begin{proof}
  By definition, $F$ is represented by $(x,x_{\_})$ iff for every $y$
  the functor $K_y \colon\catC(y,x)^{op} \to Fy$, given by $f\mapsto Ff(x_{\_})$,
  is an equivalence.

  Condition~\eqref{eq:psfunctorReprCondOne} says that each $K_y$ is essentially surjective.
  It remains to show that, for each $y$, $K_y$ is full and faithful
  iff condition~\eqref{eq:psfunctorReprCondTwo} holds.

  Suppose $K_y$ is full and faithful and, for a given $y_{\_}$,
  we have
  \[
    (f,f_{\_}), (g,g_{\_}) \colon (y,y_{\_}) \to (x,x_{\_})
  \]
  with $(f,f_{\_})$ cartesian,
  i.e. $f_{\_}$ an isomorphism.
  Then there is a unique $\alpha\colon g \to f$ such that $F\alpha_{x_{\_}} = f_{\_}^{-1};g_{\_}$,
  in other words a unique 2-cell from $(g,g_{\_})$ to $(f,f_{\_})$.

  Conversely, suppose condition~\eqref{eq:psfunctorReprCondTwo} holds for a given $y$,
  and suppose we have $f,g \colon y \to x$ and $g_{\_}\colon Ff(x_{\_}) \to Fg(x_{\_})$.
  We then have two 1-cells
  \[
    (f,\Id),(g,g_{\_})\colon (y,Ff(x_{\_})) \to (x,x_{\_})\text{.}
  \]
  Since $(f,\Id)$ is cartesian we get a unique 2-cell $\alpha \colon (g,g_{\_}) \to (f,\Id)$,
  in other words, a unique $\alpha\colon g \to f$ such that $K_y(\alpha) = g_{\-}$.
\end{proof}

By the usual means, one can show that if $x$ is a representing object for $P$,
then for any object $x'$ in $\catE$ we have that $x'$ is a representing object iff
it is equivalent to $x$.

We now extend the above discussion to a situation where $\catC$ too is fibred:
we have fibrations
\[
  \xymatrix{
    {\catE} \ar@{->}[r]^{P} & {\catC} \ar@{->}[r]^{Q} & {\catB}
  }
  \text{.}
\]
In our applications, $P$ will again be got from a pseudofunctor (in fact a 2-functor)
$\catC^{coop} \to \Cat$,
but $Q$ will be more general.
The paradigm example for $Q$ is $\GTop^{co}$ fibred over $\Top^{co}_{\cong}$.

We also assume (as there) that all 2-cells in $\catB$ are isomorphisms.

Now each object $w$ of $\catB$ has a fibre over it, a fibration $P_w\colon\catE_w\to\catC_w$:
it comprises the 0-cells of $\catC$ and $\catE$ that map to $w$, and the 1- and 2-cells
that map to identities at $w$.
We are now interested in the situation where each $P_w$ is representable,
and in how the representing objects transform under 1-cells in $\catB$.

Since we are assuming $P$ arises from a pseudofunctor,
it is easy to see that a 1-cell or 2-cell in $\catE_w$ is cartesian for $P_w$
iff it is cartesian for $P$.

\begin{definition}\label{def:locRepr}
  $P$ is \emph{locally representable} (over $Q$) iff
  \begin{enumerate}
  \item
    Each fibre $P_w$ is representable.
  \item
    \emph{(Geometricity)}
    Suppose $P_w$ is represented by $x_w$, $f\colon w' \to w$ in $\catB$,
    and $h\colon y \to x_w$ is $PQ$-cartesian over $f$.
    Then $y$ is a representing object for $P_{w'}$.
  \end{enumerate}
\end{definition}
We call condition~(2) ``geometricity'' in line with~\cite{PPExp},
because it concerns a property that is preserved by pseudopullback in $\Top$.
Note that it suffices to verify it for \emph{some} $x_w$ and \emph{some} $h$.
This is because representing objects are equivalent,
and so too are cartesian liftings.

\begin{proposition}\label{prop:locRepr}
  $P$ is locally representable over $Q$ iff, for each object $w$ of $\catB$,
  we have an object $x_w$ of $\catE$ over it that satisfies the following conditions.
  \begin{enumerate}
  \item\label{eq:locReprCondOne}
    For every object $y$ of $\catE$, and 1-cell $f\colon Q(Py) \to w$ in $\catB$,
    there is some $\hat{f}\colon y \to x_w$ over $f$ that is cartesian with respect to $P$.
  \item\label{eq:locReprCondTwo}
    Suppose $f\colon y \to x_w$ in $\catE$ is cartesian with respect to $P$.
    If $g \colon y \to x_w$, and $\alpha\colon Q(Pg) \to Q(Pf)$,
    then there is a unique $\hat{\alpha}\colon g \to f$ over $\alpha$.
  \end{enumerate}
\end{proposition}
\begin{proof}
  $\Leftarrow$:
  Clearly any $x_w$ satisfying the conditions must be a representing object for $P_w$.
  It remains to show that the representing objects transform
  correctly under base 1-cells $f\colon w' \to w$.

  Suppose $x_w$ and $x_{w'}$ satisfy the conditions.
  By the conditions for $x_w$ we have $P$-cartesian $g\colon x_{w'}\to x_w$ over $f$.
  Suppose also that $h\colon y \to x_w$ is $PQ$-cartesian over $f$.
  By the conditions on $x_{w'}$ we get $P$-cartesian $u\colon y \to x_{w'}$ over $\Id_{w'}$,
  and by cartesianness of $h$ we get $v\colon x_{w'} \to y$ over $\Id_{w'}$
  with an isomorphism $\alpha\colon vh \to g$ over $\Id_{f}$.
  \[
    \xymatrix{
      {x_{w'}}
        \ar@{->}[dr]^{g}
        \ar@{}[dr]_{\alpha\Uparrow}
        \ar@{->}[d]^{v}
      \\
      {y}
        \ar@{->}@/^1pc/[u]^{u}
        \ar@{->}[r]_{h}
      & {x_w}
    }
  \]
  Since both $g$ and $h$ are $P$-cartesian, so is $v$.
  It follows by the conditions on $x_{w'}$ that there is a unique isomorphism
  $vu\cong\Id_{x_{w'}}$ in $P_{w'}$.
  Also, by the $PQ$-cartesian property of $h$,
  there is a unique isomorphism $uv\cong\Id_{y}$ in $P_{w'}$.
  Hence $y$ is equivalent to $x_{w'}$, and so represents $P_{w'}$ as required.

  $\Rightarrow$:
  Let $x_w$ be a representing object for $P_w$.
  We show it has the two properties stated.

  Suppose $y$ is an object in $\catE$, and $f\colon w' = Q(Py) \to w$ a 1-cell in $\catB$.
  Let $g\colon x_{w'} \to x_w$ be $PQ$-cartesian over $f$,
  so that $x_{w'}$ is a representing object for $P_{w'}$.
  Then there is a $P$-cartesian 1-cell $u\colon y\to x_{w'}$ in $P_{w'}$,
  and $ug\colon y \to x_w$ is $P$-cartesian (because $u$ and $g$ are) over $f$.

  Now suppose $h_0, h_1 \colon y \to x_w$ are two $P$-cartesian 1-cells,
  with $f_i = Q(P h_i) \colon w' \to w$, and $\alpha \colon f_0 \to f_1$.
  Recall our assumption that all 2-cells in $\catB$ are isomorphisms.
  Let $g_i \colon z_i \to x_w$ be a $PQ$-cartesian lifting of $f_i$,
  with $u_i \colon y \to z_i$ and $\beta_i\colon u_i g_i \cong h_i$.
  By~\cite[3.1.15]{Buckley:Fibred2CatBicat},
  there is an equivalence $k\colon z_0 \simeq z_1$ with isomorphism $k g_1 \cong g_0$ over $\alpha$,
  and the pair is unique up to unique isomorphism between $k$s in $P_{w'}$.
  Since $z_1$ is a representing object for $P_{w'}$,
  we get a unique 2-cell $u_0 k \to u_1$ in $P_{w'}$,
  and putting these together gives the unique 2-cell $h_0 \to h_1$ over $\alpha$ as required.
\end{proof}

\section{Context extensions as spaces}\label{sec:ExtSp}
In this Section we gather together the previous remarks to get results on classifying toposes
in a form that is fibred over a category of bases.

This is most easily understood in the simple case of a single context $\thT$.
For each Grothendieck topos $p \colon \catE \to \baseS$ we have a category $\Mods{\thT}{\catE}$
of models of $\thT$ in $\catE$.
This extends to a 2-functor from $\GTop^{op} = (\GTop^{co})^{coop}$ to $\Cat$,
and its Grothendieck construction can be written as $P \colon (\GTopE{\thT})^{co} \to \GTop^{co}$.

In constructing that fibration we ignored the parts $\xymatrix{ {} \ar@{->}[r]^{p} & {\baseS}}$,
but when we bring in $\baseS$ we find that the classifying topos $\baseS[\thT]$ provides a representing
object for $P_{\baseS}$.

The main novelty here is that those representing objects transform according to Definition~\ref{def:locRepr}:
that the pseudopullback along any $f\colon\baseS_0\to\baseS_1$ preserves classifiers.
Our proof is non-trivial, and shows that the steps constructing the classifier are preserved under pseudopullback.

As mentioned in the Introduction,
we shall prove local representability more generally,
dealing not just with a single context $\thT$,
but in the relativized situation for an extension $\thT_0\thext\thT_1$.

Why extensions, and not arbitrary $H\colon\thT_1\to\thT_0$?
The main reason is the repeated use of Proposition~\ref{prop:extReindex}.

\subsection{Models for a context extension}\label{sec:ModExt}
\begin{definition}\label{def:modExt}
  Let $\thT_0\thext \thT_1$ be an extension of contexts,
  with corresponding extension map $U\colon\thT_1\to\thT_0$,
  and let $p \colon \catE \to \baseS$ be a bounded geometric morphism.
  A \emph{strict model of $U$ in $p$} is a pair $(M,N)$
  where $M$ is a strict model of $\thT_0$ in $\baseS$,
  $N$ a strict model of $\thT_1$ in $\catE$, and $NU = p^\ast M$.

  A morphism from one such strict model, $(M,N)$, to another, $(M',N')$,
  is a pair $\phi=(\phi_{-},\phi^{-})$ where
  $\phi_{-}\colon M \to M'$ and $\phi^{-}\colon N \to N'$ are homomorphisms
  and $\phi^{-}U = p^\ast\phi_{-}$.
\end{definition}

We thus get, for each $p$, a category $\Mods{U}{p}$.
It is strictly indexed over $\GTop$ in the following way.

First suppose $f$ is a 1-cell in $\GTop$, as in Definition~\ref{def:BTop}.
If $(M,N)$ is a strict model in $p_1$,
then we define a strict model $f^\ast(M,N) = (\downstairs{f}^\ast M, f^\ast N)$
\[
  \xymatrix{
    {f^\ast N}
      \ar@{<-}[r]^{\cong}
      \ar@{.}[d]
    & {\upstairs{f}^\ast N}
      \ar@{.}[d]
    \\
    {p_0^\ast \downstairs{f}^\ast M}
      \ar@{<-}[r]_{(\stairs{f})^\ast M}
    & {\upstairs{f}^\ast p_1^\ast M}
  }
\]
where the upstairs isomorphism is the unique one obtained from Proposition~\ref{prop:extReindex}.
The action extends to morphisms between strict models of $U$,
and we obtain a functor
$\Mods{U}{f} \colon \Mods{U}{p_1} \to \Mods{U}{p_0}$.

If $\alpha \colon f \to f'$ is a 2-cell in $\GTop$,
then it gives a natural transformation from $\Mods{U}{f}$ to $\Mods{U}{f'}$.
We obtain a strict 2-functor from $\GTop^{op}$ to $\Cat$.
Its Grothendieck construction is a fibration
$(\Modsnocat{U})^{co} \to \GTop^{co}$
defined as follows.

\begin{definition}\label{def:modExt2cat}
  The data for the 2-category $\Modsnocat{U}$ is defined as follows.
  In each case, a 0-, 1- or 2-cell is the corresponding item for $\GTop$,
  equipped with extra structure in the form of models of $U$.

  A 0-cell is a bounded geometric morphism $p\colon \catE \to \baseS$,
  equipped with a strict model $(M,N)$ of $U$.

  A 1-cell from $(p_0,M_0,N_0)$ to $(p_1,M_1,N_1)$
  is a 1-cell $f\colon p_0 \to p_1$ from $\GTop$,
  equipped with a homomorphism $(f_{-},f^{-}) \colon (M_0,N_0) \to f^\ast(M_1,N_1)$.

  Given 1-cells $(f,f_{-},f^{-})$ and $(f',f'_{-},f'^{-})$,
  with the same domain and codomain,
  a 2-cell from one to the other is a 2-cell $\alpha \colon f \to f'$ in $\GTop$
  such that $(f_{-},f^{-})(\alpha^\ast (M_1,N_1)) = (f'_{-},f'^{-})$.
\end{definition}

It is clear that $\Modsnocat{U}$ is a 2-category,
with a forgetful functor $F'\colon \Modsnocat{U} \to \GTop$,
and by construction $F'^{co}$ is a split fibration.
Note that --
\begin{enumerate}
\item
  A 1-cell $(f,f_{-},f^{-})$ is cartesian iff $f_{-}$ and $f^{-}$
  are isomorphisms.
\item
  A 2-cell $\alpha$ is (co-)cartesian iff $\alpha_{-}$ and $\alpha^{-}$
  are isomorphisms.
\end{enumerate}

Note the special case of a trivial extension $\thT_0 = \thT_0$.
A model of this in $p$ is simply a model $M$ of $\thT_0$ in $\baseS$,
since the corresponding model in $\catE$ has to be $p^\ast M$.
In this case we write $\Modsnocat{(\thT_0\thext \thT_0)}$.

We have an obvious forgetful functor from $\Modsnocat{U}$ to
$\Modsnocat{(\thT_0\thext \thT_0)}$,
which (or its co-dual) is almost, but not quite, a fibration.
The problem is that $\thT_0$-homomorphisms $\phi_{\_}\colon M \to M'$
do not lift to functors for the categories of $U$-models over them.
To rectify this, we restrict to isomorphisms downstairs.

\begin{definition}
  $\GTopU{U}$ is the sub-2-category of $\Modsnocat{U}$ with all the 0-cells,
  but with only the 1-cells $(f,f_{\_},f^{\_})$ for which $f_{\_}$ is an isomorphism.
  It is full on 2-cells.
\end{definition}

\begin{proposition}
  We write $P^{co}\colon\GTopU{U}\to\GTopB{\thT_0}$ for the forgetful functor.
  Then $P\colon(\GTopU{U})^{co}\to(\GTopB{\thT_0})^{co}$
  is a split fibration.
  1-cells and 2-cells are cartesian iff they are so in $\Modsnocat{U}$ fibred over
  $\Modsnocat{(\thT_0 \thext \thT_0)}$.
\end{proposition}
\begin{proof}
  It is the Grothendieck construction for the evident 2-functor
  from $(\GTopB{\thT_0})^{op}$ to $\Cat$.
\end{proof}

We now fibre over pairs $(\baseS,M)$.

\begin{definition}\label{def:TopB}
  The 2-category $\TopB{\thT}$ has structure as follows.
  A \emph{0-cell} is a pair $(\baseS, M)$ where $\baseS$ is an elementary topos and $M$
  a model of $\thT$ in $\baseS$.
  A \emph{1-cell} from $(\baseS_0, M_0)$ to $(\baseS_1, M_1)$
  is a pair $(\underline{f},f_{\_})$ where $f\colon\baseS_0\to\baseS_1$ is a geometric morphism
  and $f_{\_}\colon M_0 \to \underline{f}^{\ast}M_1$ is an isomorphism.
  A \emph{2-cell} from $(\underline{f},f_{\_})$ to $(\underline{g},g_{\_})$
  is a natural isomorphism $\alpha\colon\underline{f}\to\underline{g}$
  such that $f_{\_} \alpha^{\ast}M_1 = g_{\_}$.

  The 2-category $\GTopB{\thT}$ is made from $\GTop$ by adding components $M$ and $f_{\_}$,
  and the condition on $\alpha$,
  in the same way as $\TopB{\thT}$ is made from $\Top_{\cong}$.
\end{definition}

\begin{proposition}\label{prop:GTopBTFib}
  Let $Q^{co}\colon \GTopB{\thT} \to \TopB{\thT}$ be the evident forgetful functor.
  Then $Q = (Q^{co})^{co}$ is a fibration of bicategories.
\end{proposition}
\begin{proof}
  Much as in Proposition~\ref{prop:GTopFib}.
\end{proof}

We now get a diagram of 2-functors as follows,
where the $P$s and $Q$s are fibrations.
The left hand stack is for the relativized situation $\thT_0\thext\thT_1$,
while the right hand stack is the spacial case $\thT_0 = \thone$.

\begin{equation}\label{eq:mainFibs}
  \xymatrix{
    {(\GTopU{U})^{co}}
      \ar@{->}[d]_{P}
      \ar@{->}[dr]
    \\
    {(\GTopB{\thT_0})^{co}}
      \ar@{->}[d]_{Q}
      \ar@{->}[dr]
    & {(\GTopU{\thT_1})^{co}}
      \ar@{->}[d]^{P}
    \\
    {(\TopB{\thT_0})^{co}}
      \ar@{->}[dr]
    & {\GTop^{co}}
      \ar@{->}[d]^{Q}
    \\
    & {\Top_{\cong}^{co}}
  }
\end{equation}

\subsection{Context extensions fibred over models}\label{sec:ContExtFibMod}
Our aim now is to show that, in diagram~\eqref{eq:mainFibs},
each $P$ is locally representable over its $Q$.
(Note that the right hand one is a special case of the left hand,
for when $\thT_0 = \thone$.)
The existence of the representing objects (as classifying toposes) is straightforward;
what seems more novel is their preservation by pseudopullback.

\begin{proposition}
  Let $\thT_0\thext \thT_1$ be a context extension.
  Then it is also a geometric extension over any elementary topos $\baseS$.
\end{proposition}
\begin{proof}
  It suffices to check the different kinds of simple context extension.
  Note that any node $X$ in $\thT_0$ gives a context homomorphism $\thob\thmorph \thT_0$,
  and hence a geometric construct on $\thT_0$.
  Likewise, any edge or composite of edges gives a morphism between geometric constructs.

  An extension by primitive node is a geometric extension by primitive sort.

  A simple functional extension of contexts (adjoining a primitive edge) is also
  a simple functional extension of geometric theories.

  An extension by a universal is essentially no geometric extension at all,
  as the categories of (strict) models are isomorphic.

  An extension by commutativities is a simple geometric quotient,
  as imposing an equality between morphisms is equivalent to requiring the equalizer
  to be an isomorphism.
\end{proof}

\begin{proposition}
  Let $\thT_0$ be a context, and $M$ a strict model of $\thT_0$ in an elementary topos $\baseS$.
  Then there is an elephant morphism $M\colon\thone \to \thT_0$ that, on $\baseS$-topos $\catE$,
  takes $\ast$ to $p^\ast M$.
\end{proposition}
\begin{proof}
  Although the elephant theories for both $\thone$ and $\thT_0$ are strictly indexed,
  $M$ is not a strict morphism.
  Consider a morphism of $\baseS$-toposes
  \[
    \xymatrix{
      {\catF}
        \ar@{->}[rr]^{f}
        \ar@{->}[rd]_{q}
      & {}
        \ar@{}[d]|{\Downarrow\alpha}
      & {\catE}
        \ar@{->}[ld]^{p}
      \\
      & {\baseS}
    }
    \text{,}\quad
    \xymatrix{
      {\thone(\catF)}
        \ar@{=}[r]
        \ar@{->}[d]_{M(\catF)}
      & {\thone(\catE)}
        \ar@{->}[d]^{M(\catE)}
      \\
      {\thT_0(\catF)}
        \ar@{<-}[r]_{\thT_0(f)}
      & {\thT_0(\catE)}
    }
  \]
  On the right is a pseudo-naturality square,
  subject to the isomorphism $\alpha M \colon f^\ast p^\ast M \cong q^\ast M$.
\end{proof}

\begin{definition}\label{def:EoverM}
  Let $\thT_0 \thext \thT_1$ be a context extension and $M$ a strict model of $\thT_0$
  in an elementary topos $\baseS$.
  By Proposition~\ref{prop:elephantpb} we can pull back the geometric extension
  for $\thT_0\thext \thT_1$ along $M\colon\thone\to \thT_0$,
  getting a geometric theory $\thT_1/M$ over $\baseS$.
  It has a classifying topos $\baseS[\thT_1/M]$.
\end{definition}

\begin{theorem}\label{thm:pspbClassifiers}
  Let $\thT_0 \thext \thT_1$ be a context extension and $M$ a strict model of $\thT_0$
  in an elementary topos $\baseS_1$.
  Let the following diagram be a pseudopullback in $\Top$.
  \[
    \xymatrix{
      {\catE}
        \ar@{->}[r]^-{\upstairs{f}}
        \ar@{->}[d]_{p_0}
        \ar@{}[dr]|{f\Downarrow}
      & {\baseS_1[\thT_1/M]}
        \ar@{->}[d]^{p_1}
      \\
      {\baseS_0}
        \ar@{->}[r]_{\downstairs{f}}
      & {\baseS_1}
    }
    \text{,}\quad
    \stairs{f} \colon\upstairs{f}p_1 \cong p_0\downstairs{f}
    \text{.}
  \]
  Then $p_0\colon \catE\to\baseS_0$ serves as a classifying topos
  $\baseS_0[\thT_1/\downstairs{f}^\ast M]$.
\end{theorem}
\begin{proof}
  First, pseudopullback squares are preserved under composition with equivalences over
  $\baseS_0$ and $\baseS_1$,
  so it suffices to show that there is \emph{some} pseudopullback square
  whose vertical maps are classifiers as stated.

  Suppose the extension is of the form $\thT_0\thext \thT'_1\thext \thT_1$
  where $\thT'_1\thext \thT_1$ is simple,
  with extension maps
  $\xymatrix{{\thT_1} \ar@{->}[r]^{U'} & {\thT'_1} \ar@{->}[r]^{U} & {\thT_0}}$.
  Then by induction we can assume the result holds for $\thT_0\thext \thT'_1$ and
  examine the cases for $\thT'_1\thext \thT_1$.

  For extension by primitive node, we have the task of constructing an object classifier,
  and this is a special case of classifying torsors over an internal category $\catC$,
  here the category of finite sets:
  objects are natural numbers, morphisms defined in the appropriate way.

  For extension by commutativity, we have already remarked that this is equivalent to
  inverting a morphism.

  For a simple functional extension, adjoining a morphism from $X$ to $Y$,
  we can decompose the classification problem into two steps of the above kinds.
  First, we adjoin a subjobject of $X\times Y$,
  and this is equivalent to adjoining a torsor (ideal) for the poset $\fin(X\times Y)$.
  Next we impose some axioms for single-valuedness and totality,
  and this is equivalent to making some morphism invertible.

  It follows that we reduce to two cases over $\thT'_1$:
  adjoining a torsor for an internal category $\catC$,
  and forcing the invertibility of some morphism.
  (Although these are not simple extensions of contexts,
  we can still work with them as single steps.)

  Suppose we have a classifier $p\colon \baseS' = \baseS[\thT'_1/M] \to \baseS$,
  with generic model $N'_G$,
  and assume the classifier $p' \colon \baseS'[\thT_1/N'_G] \to \baseS'$ exists,
  with generic model $N_G$.
  Note that, using Lemma~\ref{lem:fMH}, $N_G U'U = p'^\ast N'_G U = (p'p)^\ast M$.
  \[
    \xymatrix{
      {\baseS'' = \baseS'[\thT_1/N'_G]}
        \ar@{->}[d]^{p'}
      \\
      {\baseS' = \baseS[\thT'_1/M]}
        \ar@{->}[d]^{p}
      \\
      {\baseS}
    }
  \]

  We show that $\baseS''$ classifies $\thT_1/M$.
  For the ``essential surjectivity'' part, suppose $N$ is a model of $\thT_1/M$ in $(\catF,q)$.
  Then $NU'$ is a model of $\thT'_1/M$,
  so we get $g = (\upstairs{g},\alpha)\colon(\catF,q)\to(\baseS',p)$ with
  $NU'\cong g^\ast N'_G \cong \upstairs{g}^\ast N'_G$.
  Now using Proposition~\ref{prop:extReindex} we can find a model $N'\cong N$ of $\thT_1$
  with $N'U' = \upstairs{g}^\ast N'_G$, so $N'$ is a model of $\thT_1/N'_G$ in $(\catF,\upstairs{g})$.
  We now get the required geometric morphism to $\baseS''$.

  Now suppose we have two morphisms
  $f_i =(\upstairs{f}_i,\alpha_i)\colon(\catF,q)\to (\baseS'',p'p)$ ($i=0,1$)
  and a morphism $\beta\colon f_0^\ast N_G \to f_1^\ast N_G$.
  Let us write $q'_i = \upstairs{f}_i p'$,
  so that $\catF$ becomes two distinct toposes $(\catF,q'_i)$ over $\baseS'$.
  Using Lemma~\ref{lem:fMH} we find $q'^\ast_i N'_G = f_i^\ast p'^\ast N'_G = f_i^\ast N_G U$,
  and so we get $\beta U \colon q'^\ast_0 N'_G \to q'^\ast_1 N'_G$.
  We obtain a unique $\gamma'\colon q'_0 \to q'_1$ such that
  $\gamma'^\ast N'_G = \beta U$,
  and this gives us a geometric morphism
  $q' = \langle q'_0, \gamma', q'_1 \rangle \colon [\to,\catF] \to \baseS'$
  where ``$\to$'' here denotes the category with two objects and three morphisms.
  Let us write $h_i \colon \catF \to [\to,\catF]$ for the two geometric morphisms
  whose inverse image parts are the domain and codomain functors,
  and $\eta\colon h_0 \to h_1$ for the corresponding natural transformation.
  Then $h_i q' \cong q'_i$, and in fact we have equality on the inverse image parts.

  $([\to,\catF],q')$ has a model $N$ of $\thT_1$ given by
  $ \xymatrix{ {f_0^\ast N_G} \ar@{->}[r]^{\beta} & {f_1^\ast N_G}}$.
  It is not a model of $\thT_1/N'_G$,
  since $q'^\ast N'_G$ uses $\upstairs{f}_i^\ast p^\ast N'_G$,
  which are only isomorphic to $f_i^\ast p'^\ast N'_G = f_i^\ast N_G U$
  (see Section~\ref{sec:ModExt}).
  However, by Proposition~\ref{prop:extReindex} we can find a model
  $N' = (\xymatrix{ {N'_0} \ar@{->}[r]^{\beta'} & {N'_1}})$ of $\thT_1$ isomorphic to $\beta$,
  and whose $\thT'_1$ reduct is $q'^\ast N'_G$.
  From this we get $g = (\upstairs{g},\delta) \colon ([\to,\catF],q') \to (\baseS'',p')$
  such that $g^\ast N_G \cong N'$ over $q'^\ast N'_G$.
  \[
    \xymatrix{
      & {[\to,\catF]}
        \ar@{->}[dr]^{g}
        \ar@{->}[dd]^(0.3){q'}
      \\
      {\catF}
        \ar@{->}[ur]^{h_i}
        \ar@{->}[rr]^(0.3){\upstairs{f}_i}
        \ar@{->}[dr]_{q'_i}
      & & {\baseS''}
        \ar@{->}[dl]^{p'}
      \\
      & {\baseS'}
    }
  \]
  Now for each $i$ we find a $\thT_1$-isomorphism
  $h_i^\ast g^\ast N_G \cong h_i^\ast N' \cong \upstairs{f}_i^\ast N_G$
  over $q'^\ast_i N'_G$,
  and so we have a unique isomorphism $h_i g \cong \upstairs{f}_i$ giving rise to it,
  and it must be the identity over $\baseS$.
  Putting these together with $\eta g \colon h_0 g \to h_1 g$ gives our required
  $\gamma\colon f_0 \to f_1$.

  Now it remains only to show that our classifiers $\baseS'[\thT_1/N'_G]$ can be found
  in a way that is preserved under pseudopullback.
  The argument parallels that of~\cite[B3.3.6]{Elephant1}.

  In one case, $\thT_1$ adjoins a torsor for an internal category $\catC$ in $\baseS'$.
  Here we can take the classifier to be $[\catC,\baseS']$
  by Diaconescu's Theorem, and this can be pulled back along
  any $g\colon\catF\to\baseS'$ to $[g^\ast\catC,\catF]$ over $\catF$.
  (See~\cite[B3.2.7, B3.2.14]{Elephant1}.)

  In the other case, $\thT_1$ imposes an invertibility for a morphism $u\colon X\to Y$
  in $\baseS'$.
  Here $p'\colon \baseS'' \to \baseS'$ is an inclusion,
  and by~\cite[A4.3.11]{Elephant1} it can be taken to be the topos of sheaves
  for the smallest local operator for which $\mathrm{im}(u) \rightarrowtail Y$
  and $X \rightarrowtail \mathrm{kp}(u)$, the kernel pair, are both dense.
  For inverting both of these monomorphisms will make $u$ invertible.
  By~\cite[A4.5.14(e)]{Elephant1} its pseudopullback along $g$ is also an inclusion,
  in fact for the smallest local operator that makes $g^\ast u$ an isomorphism.
\end{proof}

\begin{proposition}\label{prop:classTopGen}
  Let $\thT_0 \thext \thT_1$ be a context extension and $M$ a strict model of $\thT_0$
  in an elementary topos $\baseS$.
  Then $\baseS[\thT_1/M]$ has the classifying topos property for arbitrary
  $q\colon\catF\to\baseS$, not necessarily bounded.
\end{proposition}
\begin{proof}
  If we have $q\colon \catF \to \baseS$, and a model $N$ of $\thT_1/q^\ast M$,
  then by Theorem~\ref{thm:pspbClassifiers} we can make a diagram
  \[
    \xymatrix{
      {\catF}
        \ar@{->}[r]^-{\upstairs{g}}
        \ar@{=}[dr]^{\beta\Downarrow}
      & {\catF[\thT_1/q^\ast M]}
        \ar@{->}[r]^{\upstairs{q}}
        \ar@{->}[d]^{p'}
        \ar@{}[dr]|{\alpha\Downarrow}
      & {\baseS[\thT_1/M]}
        \ar@{->}[d]^{p}
      \\
      & {\catF}
        \ar@{->}[r]_{q}
      & {\baseS}
    }
    \text{,}
  \]
  where the square is a pseudopullback and $g^\ast N'_G \cong N$.

  On closer examination we find that
  $(\upstairs{g},\beta)(\upstairs{q},\alpha)$ provides a suitable morphism from $\catF$
  to $\baseS[\thT_1/M]$ over $\baseS$ as required for $N$,
  and also we have the appropriate fullness and faithfulness conditions.%
\end{proof}

Putting together these results, we now obtain --

\begin{theorem}\label{thm:locRep}
  In diagram~\eqref{eq:mainFibs}, the left hand fibration $P$ is locally representable
  over its $Q$.
\end{theorem}

As we have already mentioned, by taking $\thT_0 = \thone$ we get that the right hand $P$
is also locally representable.

\begin{example}\label{ex:geomPres}
  Let $\thT_0$ be the context whose models are ``GRD-systems'' as in~\cite{PPExp}.
  It has three nodes $G,R,D$, together with (amongst other ingredients)
  a further node $\fin G$ constrained to be the Kuratowski finite powerset of $G$.
  (For instance, it can be constructed as a quotient of the list object $\List G$.)
  Finally, it has edges
  \[
    \xymatrix{
      & {D}
        \ar@{->}[dl]_{\rho}
        \ar@{->}[d]^{\pi}
      \\
      {\fin G}
      & {R}
        \ar@{->}[l]^{\lambda}
    }
  \]
  This can be used to present a locale, with generators $g\in G$ subject to relations (for $r\in R$)
  \[
    \bigwedge \lambda(r) \leq \bigvee \{ \bigwedge \rho(d) \mid \pi(d) = r \}
    \text{.}
  \]

  The points of the locale, the subsets $F \subseteq G$ respecting the relations,
  are models of a context $\thT_1$ that extends $\thT_0$.
  It has a node for $F$, with an edge $F\to G$ constrained to be monic,
  nodes for $X = \{ r\in R \mid \lambda(r) \subseteq F \}$
  and $Y = \{ r \in R \mid (\exists d) (\pi(d) = r \wedge \rho(d) \subseteq F) \}$
  (which can be constructed in the AU-sketches)
  and an edge $X \subseteq Y$.

  Then the local representability Theorem~\ref{thm:locRep} implies~\cite[Corollary~5.4]{PPExp},
  the geometricity of presentations.
\end{example}

\section{Conclusion}\label{sec:Conc}
What we have done here is to elaborate the idea that a map $U\colon\thT_1\to\thT_0$,
a $\thT_0$-valued map on $\thT_1$,
may also be a \emph{bundle:}
that is to say, a \emph{space}-valued map on the \emph{co}domain $\thT_0$,
transforming points to the corresponding fibres.

This interpretation is often tacit in a morphism in a category,
and is particularly important in type theory.
We have made it concrete in the particular case of a morphism $U$ in $\Con$
that arises from a context extension.

Note that $U$ certainly is a ``$\thT_0$-valued map on $\thT_1$'',
if we think of the points of a context as its strict models.
This is shown in Section~\ref{sec:IndCatMod} and does not need toposes --
the models can be taken in any AU.

To get $U$ as a bundle, we interpret ``space'' as Grothendieck topos
and look for the classifying toposes for the fibres.
However, the base toposes are now allowed to vary,
and in Theorem~\ref{thm:pspbClassifiers} we showed the geometricity property
that when you change the base,
and the corresponding base point of $\thT_0$,
the classifier (representing the fibre) transforms by pseudopullback.
This result, which I have not been able to find in the literature,
relies on a difference between the ``arithmetic'' theories of $\Con$
and the geometric theories that are classified.
An arithmetic theory depends only on the existence of an nno,
whereas a geometric theory depends on the choice of some base topos $\baseS$.

To avoid the intricacies of coherence for the choices made in indexed categories,
we have adopted a fibrational approach to classifiers.

The results here are a piece in the broad programme of using AU techniques
to prove base-independent, geometric results for toposes in those situations that
do not need the full power of $\baseS$-indexed colimits for some $\baseS$.
One already mentioned is the ``geometricity of presentations'', Example~\ref{ex:geomPres}.

On the other hand, the results also provide clues to how one might seek a self-standing arithmetic logic
of spaces, developing~\cite{ArithInd}.
They suggest that the extension maps might be the correct analogues of bounded geometric morphisms.

\section*{Acknowledgements}
I am grateful to the University of Birmingham for a period of study leave that enabled me
to bring this paper to completion;
also to Thomas Streicher for various helpful discussions that clarified my understanding of fibrations.

\bibliographystyle{amsalpha}
\bibliography{../bibtex/MyBiblio}

\end{document}